\documentclass[11pt,reqno]{article}

\usepackage[usenames]{color}
\usepackage{amssymb}
\usepackage{graphicx}
\usepackage{amscd}
\usepackage{multirow}
\usepackage{epstopdf}
\usepackage{color}
\usepackage{fullpage}
\usepackage{float}
\usepackage{graphics,amsmath,amssymb}
\usepackage{amsthm}
\usepackage{amsfonts}
\usepackage{latexsym}
\usepackage{hyperref}
\usepackage{caption}
\usepackage{slashbox}
\usepackage{booktabs}

\newcommand{\cln}[2]{ \left\lceil \frac{#1}{#2} \right\rceil}
\newcommand{\flr}[2]{ \left\lfloor \frac{#1}{#2} \right\rfloor}

\theoremstyle{plain}
\newtheorem{theorem}{Theorem}[section]

\newtheorem{conjecture}[theorem]{Conjecture}

\theoremstyle{definition}

\theoremstyle{remark}

\numberwithin{equation}{section}
\numberwithin{theorem}{section}
\numberwithin{table}{section}
\numberwithin{figure}{section}

\begin{document}

\begin{center}
\vskip 1cm{\LARGE\bf
Nested Recursions with Ceiling Function Solutions\\
}
\vskip 0.1 in
\large
Abraham Isgur\\
Vitaly Kuznetsov\\
Stephen M. Tanny\\
Department of Mathematics\\
University of Toronto\\
Toronto, Ontario M5S 2E4\\
Canada\\
\href{mailto:abraham.isgur@math.utoronto.ca}{\tt abraham.isgur@math.utoronto.ca}\,,\href{mailto:v.kuznetsov@utoronto.ca}{\tt v.kuznetsov@utoronto.ca}\,, \href{mailto:tanny@math.utoronto.ca}{\tt tanny@math.utoronto.ca}\\
\end{center}

\vskip .1 in

\begin{abstract}
Consider a nested, non-homogeneous recursion $R(n)$ defined by $$R(n) = \sum_{i=1}^k R\left(n-s_i-\sum_{j=1}^{p_i} R(n-a_{ij})\right) + \nu,$$ with $c$ initial conditions $R(1) = \xi_1 > 0,R(2)=\xi_2 > 0,
\ldots, R(c)=\xi_c > 0$, where the parameters are integers satisfying $k > 0, p_i > 0$ and $a_{ij} > 0$. We develop an algorithm to answer the following question: for an arbitrary rational number $r/q$, is there any set of values for $k, p_i, s_i, a_{ij}$ and $\nu $ such that the ceiling function $\cln{rn}{q}$ is the unique solution generated by $R(n)$ with appropriate initial conditions? We apply this algorithm to explore those ceiling functions that appear as solutions to $R(n)$. The pattern that emerges from this empirical investigation leads us to the following general result: every ceiling function of the form $\cln{n}{q}$ is the solution of infinitely many such recursions. Further, the empirical evidence suggests that the converse conjecture is true: if $\cln{rn}{q}$ is the solution generated by any recursion $R(n)$ of the form above, then $r=1$. We also use our ceiling function methodology to derive the first known connection between the recursion $R(n)$ and a natural generalization of Conway's recursion.
\end{abstract}


\section{Introduction} \label{sec:Intro}
In this paper we investigate the occurrence of ceiling function solutions to nested, non-homogeneous\footnote{Golomb \cite{Golomb1990} first solved the simplest example of such a non-homogeneous nested recursion, namely, $G(n)=G(n-G(n-1))+1$, $G(1)=1$; see also \cite{Rpaper}. In fact, all of the recursions we find with ceiling function solutions can be rewritten into an equivalent homogeneous form. We use the non-homogeneous form for two reasons; first, we do not \textit{a priori} know we will find no others, and second, it simplifies the presentation of the material in Section \ref{sec:Alg}.} recursions of the form
\begin{align}
R(n) = \sum_{i=1}^k R\left(n-s_i-\sum_{j=1}^{p_i} R(n-a_{ij})\right) + \nu
\label{eq:generalRecurrence}.
\end{align}
Unless otherwise noted, we consider only $n > 0$. The parameters in (\ref{eq:generalRecurrence}) are all integers satisfying $k, p_i$ and $a_{ij}>0$. Assume $c$ initial conditions $R(1) = \xi_1,R(2)=\xi_2,
\ldots, R(c)=\xi_c$, with all $\xi_i>0$. Following \cite{ConollyLike}, where the homogeneous version of (\ref{eq:generalRecurrence}) is first introduced in full generality, we denote such recursions by
\begin{align}
\label{eq:generalNotation}
\langle s_1;a_{11},a_{12},\ldots,a_{1p_1} : s_2;a_{21},a_{22},\ldots,a_{2p_2} :
  \cdots :s_k;a_{k1},a_{k2},\ldots,a_{kp_k}|\nu \rangle [\xi_1,\xi_2,\ldots,\xi_c ].
\end{align}

In \cite{ConollyLike}, initial results concerning ceiling function solutions of (\ref{eq:generalRecurrence}) are obtained in the special case $k=2$, $p_1=p_2=p$, and $\nu=0$; specifically, necessary and sufficient conditions are proven on the $s_i$ and $a_{ij}$ for $\cln{n}{2p}$ to solve (\ref{eq:generalRecurrence}) (see \cite{ConollyLike}, Section 5, Theorem 5.2).

Here we focus on a far more general question: can we identify those values of the parameters in (\ref{eq:generalRecurrence}) for which the solution is any ceiling function of the form $\cln{rn}{q}$, where $r/q$ is an arbitrary rational number?  From Theorem 2.1 in \cite{ConollyLike} it is immediate that $\frac{r}{q} = \frac{k-1}{\sum_{i=1}^{k}p_i}$ is a necessary condition for such a ceiling function solution to (\ref{eq:generalRecurrence})\footnote{Note that while the proof of Theorem 2.1 in \cite{ConollyLike} refers to a homogeneous recursion, it is evident that the same argument holds for a non-homogeneous recursion. Since $\frac{\nu}{n} \rightarrow  0$ as $n \rightarrow \infty$, the non-homogeneous term drops out of the limit of $\frac{A(n)}{n}$; the remainder of the proof of Theorem 2.1 establishes the desired result for non-homogeneous $R(n)$.}. For any given $r/q$, this somewhat restricts the possible values of $k$ and $p_i$ for which a ceiling function solution could occur. However, it turns out to be much more natural and productive to explore the related question: for any given set of values of $k$ and $p_i$, is there a set of values for $s_i$ and $a_{ij}$ such that $\cln{rn}{q}$ solves (\ref{eq:generalRecurrence}), where $\frac{r}{q} = \frac{k-1}{\sum_{i=1}^{k}p_i}$?

In Section \ref{sec:Alg} we answer this latter question completely for any given, fixed set of parameters $k$ and $p_i$. To do so we first show how to reduce the possible choices for the unknown parameters to a finite set, and then provide a computer-based approach for checking each possibility. Further, we show that for any given $r/q$ with $r>1$, only finitely many different combinations of the parameters $k$ and $p_i$ are possible. In Section \ref{sec:Empirical} we apply this approach and analyse the findings. The empirical results inspire one new theorem that identifies for every $q$ an explicit recursion solved by $\cln {n}{q}$, as well as several conjectures. In \ref{sec:Conway} we show how our study of rational ceiling function solutions to (\ref{eq:generalRecurrence}) leads to a discovery that these same functions appear as solutions to a natural generalization of Conway's recursion (for background on Conway's recursion, see \cite{Conway, KuboVakil}). We conclude in Section \ref{sec:Conc} with suggestions for future work.

\section{Reduction to Finitely Many Cases} \label{sec:Alg}

We begin by proving that $\cln{rn}{q}$ satisfies (\ref{eq:generalRecurrence}) for all $n$ if and only if it satisfies (\ref{eq:generalRecurrence}) for $q^2$ consecutive values of $n$. For technical reasons, we need to distinguish between the property that a nested recursion $R(n)$ with given initial conditions generates $B(n)$ as its (unique) solution sequence via a recursive calculation, and the property that the sequence $B(n)$ \emph{formally satisfies} the recursion $R(n)$\footnote{See \cite{ConollyLike}, where the idea of formal satisfaction is discussed.}. By the latter we mean only that for all $n$, $B(n)$ satisfies the equation that defines $R(n)$, in the sense that if we substitute the appropriate values of the $B$ sequence on both sides of the equation we get the required equality. But it may \textit{not} be the case that $R(n)$ generates $B(n)$ as its unique solution.

For example, the sequence $\cln{n}{2}$ formally satisfies the recursion $R(n) = R(n+1-R(n+1))-R(n-2-R(n-3))$, but $R(n)$ does not generate $\cln{n}{2}$ for any set of initial conditions because the recursion for $R(n)$ requires that we know the term $R(n+1)$ to calculate $R(n)$. In this example the recursion $R(n)$ has some negative parameters, a situation that we don't normally permit. But we can still find examples of formal satisfaction without generation as a solution where all the parameters are positive, although such examples are necessarily more complex. The reader can verify using Theorem \ref{thm:nIsBounded} below that $\cln{2n}{3}$ formally satisfies the recursion $\langle 4; 1: 3; 2: 9; 5: 7; 4: 6; 1: 7; 5: 6; 15, 15, 15| 0 \rangle$. But this recursion does not generate $\cln{2n}{3}$ as a solution no matter how many initial conditions we provide that equal the function. The problem arises in the last term. Calculating $R(n)$ requires evaluating $R(n-6-3R(n-15))$. But for large enough $n$ each $R(n-15)$ would be about $2n/3$, so $n-6-3R(n-15)$ would be about $-n$. Thus we cannot evaluate $R(n-6-3R(n-15))$. We explain the conditions under which this latter phenomenon occurs in Theorem \ref{thm:FormalGeneration} below.

\begin{theorem}
For some fixed $r/q$, the recursion (\ref{eq:generalRecurrence}) is formally satisfied by $\cln{rn}{q}$ for all $n$ if and only if it is satisfied by $\cln{rn}{q}$ for $0< n \leq q^2$.
\label{thm:nIsBounded}
\end{theorem}

\begin{proof}
The ``only if" condition is obvious. We prove necessity.
First we show that if (\ref{eq:generalRecurrence}) is formally satisfied by
$\cln{rn}{q}$ for some given value of $n$, then it is also formally satisfied by this ceiling function at $n+q^2$.
If (\ref{eq:generalRecurrence}) is formally satisfied by
$\cln{rn}{q}$, then:
\begin{align*}
\cln{rn}{q} = \sum_{i=1}^k \cln{r(n-s_i-\sum_{j=1}^{p_i} \cln{r(n-a_{ij})}{q})}{q} + \nu
\end{align*}
We now evaluate both sides of the above equation at $n+q^2$. The left hand side is easy: $\cln{r(n+q^2)}{q} = \left\lceil\frac{rn}{q} + rq\right\rceil = \cln{rn}{q}+rq$. The right hand side requires a little more effort:
\begin{align*}
\sum_{i=1}^k \cln{r(n+q^2-s_i-\sum_{j=1}^{p_i} \cln{r(n+q^2-a_{ij})}{q})}{q} + \nu=
\sum_{i=1}^k \cln{r(n+q^2-s_i-\sum_{j=1}^{p_i} (\cln{r(n-a_{ij})}{q}+rq))}{q}+ \nu\\
=\sum_{i=1}^k \cln{r(n+q^2-rqp_i-s_i-\sum_{j=1}^{p_i} \cln{r(n-a_{ij})}{q})}{q}+ \nu\\
=\sum_{i=1}^k (\cln{r(n-s_i-\sum_{j=1}^{p_i} \cln{r(n-a_{ij})}{q})}{q} + rq-r^2p_i)+ \nu\\
=\sum_{i=1}^k \cln{r(n-s_i-\sum_{j=1}^{p_i} \cln{r(n-a_{ij})}{q})}{q} + krq-r^2\sum_{i=1}^{k} p_i+ \nu.
\end{align*}

As we mentioned in Section \ref{sec:Intro}, since $\cln{rn}{q}$ satisfies (\ref{eq:generalRecurrence}), it must be the case that $\frac{r}{q} = \frac{k-1}{\sum_{i=1}^{k}p_i}$. Thus,
$\sum_{i=1}^{k}p_i = \frac{q(k-1)}{r}$, from which we get $krq-r^2\sum_{i=1}^{k} p_i=krq-(k-1)rq=rq$, and both sides of the equation agree, as required. So we have shown that if (\ref{eq:generalRecurrence}) is formally satisfied by
$\cln{rn}{q}$ for some $n$, then it is formally satisfied by this ceiling function
for $n+q^2$. Note that same argument shows that this is also the case for $n-q^2$ instead of $n+q^2$. Thus, by the usual division algorithm, we conclude that
if (\ref{eq:generalRecurrence}) is satisfied by $\cln{rn}{q}$ for $0< n \leq q^2$,
then it is formally satisfied by this ceiling function for all $n$. This completes the proof.
\end{proof}

Theorem \ref{thm:nIsBounded} makes it possible to check by direct calculation whether the ceiling function $\cln{rn}{q}$
formally satisfies a given nested recursion (\ref{eq:generalRecurrence}) with $\frac{r}{q} = \frac{k-1}{\sum_{i=1}^{k}p_i}$; we merely check that equality holds for finitely many values of $n$. However, this still leaves infinitely many combinations of the parameters $s_i, a_{ij}$ and $\nu$ to check in order to answer the question we posed at the outset. Thus, our next step is to reduce this parameter space to a finite number of combinations.

Given (fixed) $k$ and $p_i$, we establish a natural equivalence relation on the set of combinations of parameters $s_i, a_{ij}$ and $\nu$. For an arbitrary combination, treat the tuple
$\langle s_1;a_{11},a_{12},\ldots,a_{1p_1} : s_2;a_{21},a_{22},\ldots,a_{2p_2} :
  \cdots :s_k;a_{k1},a_{k2},\ldots,a_{kp_k}|\nu \rangle$
as a vector, denoted by $  y $. For simplicity here we allow negative $a_{ij}$, but we will show that every equivalence class has a member with all $a_{ij}>0$; see Theorem \ref{thm:boundedParameters}. Note that $  y \in \mathbb{Z}^{k+1+\sum_{i=1}^{k} p_{i}}$.

Define the difference function $ h(n,  y) = \cln{rn}{q} - \sum_{i=1}^{k} \cln{r(n-s_{i}-\sum_{j=1}^{p_{i}}\cln{r(n-a_{ij})}{q}}{q}-\nu$. Observe that for fixed $  y$, (\ref{eq:generalRecurrence}) is formally satisfied by $\cln{rn}{q}$ if and only if $h(n,  y)=0$. We define the following equivalence relation on the vectors $  y$:

\begin{gather*}
 \langle \cdots : s_i; \ldots, a_{ij}, \ldots : \cdots \rangle
 \sim
 \langle \cdots : s_i+dr; \ldots, a_{ij}+dq, \ldots : \cdots \rangle \tag{i} \\
  \langle \cdots : s_i; \ldots : \cdots | \nu \rangle
 \sim
 \langle \cdots : s_i+cq; \ldots :  \cdots | \nu+cr \rangle \tag{ii}
\end{gather*}
where $c$, $d \in \mathbb{Z}$. Then we have the following:

\begin{theorem}
 If $  y \sim   y'$, then $h(n,  y) = h(n,   y')$.
\label{thm:invarianceOfH}
\end{theorem}

\begin{proof}

Note that it is sufficient to prove the statement of the theorem for (i) and (ii) separately.
We first check (i). Observe that

\begin{align*}
\cln{r(n-s_{i}-dr- ... -\cln{r(n-a_{ij} - dq)}{q}-...)}{q} &=
\cln{r(n-s_{i}-dr- ... -\left\lceil \frac{r(n-a_{ij})}{q} - dr \right\rceil-...)}{q} =\\
& =  \cln{r(n-s_{i}-dr-...-\cln{r(n-a_{ij})}{q}+dr-...)}{q}=\\
& = \cln{r(n-s_{i}-...-\cln{r(n-a_{ij})}{q}-...)}{q}
\end{align*}
Since the other terms in $h$ are not affected by (i), it follows that
\begin{align*}
h(n,\langle \cdots : s_i; \ldots, a_{ij}, \ldots : \cdots \rangle) =
h(n,\langle \cdots : s_i+dr; \ldots, a_{ij}+dq, \ldots : \cdots \rangle)
\end{align*}
as required. Now we check (ii). Notice that

\begin{align*}
 \cln{r(n-s_{i}-cq-\sum_{j=1}^{p_{i}} \cln{r(n-a_{ij})}{q})}{q} + \nu + cr &=
 \left\lceil\frac{r(n-s_{i}-\sum_{j=1}^{p_{i}} \cln{r(n-a_{ij})}{q})}{q} -rc \right\rceil + \nu + cr = \\
=&\cln{r(n-s_{i}-\sum_{j=1}^{p_{i}} \cln{r(n-a_{ij})}{q})}{q}-rc + \nu + cr = \\
=&\cln{r(n-s_{i}-\sum_{j=1}^{p_{i}} \cln{r(n-a_{ij})}{q})}{q} + \nu
\end{align*}
which implies that

\begin{align*}
h(n,\langle \cdots : s_i; \ldots| \nu \rangle) =
h(n,\langle \cdots : s_i+cq;\ldots| \nu+cr \rangle),
\end{align*}
completing the proof.
\end{proof}

We now show that every equivalence class has a representative with all parameters except for $\nu$ in a bounded range.

\begin{theorem}
        If $[  y] \in (\mathbb{Z}^{k+1+\sum_{i=1}^{k} p_{i}})/\sim$ there exists $  y' \in [  y]$ such that $  y'=\langle s'_1;a'_{11},a'_{12},\ldots,a'_{1p_1} : s'_2;a'_{21},a'_{22},\ldots,a'_{2p_2} :
  \cdots :s'_k;a'_{k1},a'_{k2},\ldots,a'_{kp_k}|\nu' \rangle $, with $0 \leq s'_i, a'_{ij} < q$.

\label{thm:boundedParameters}
\end{theorem}

\begin{proof}
Let $  y = \langle s_1;a_{11},a_{12},\ldots,a_{1p_1} : s_2;a_{21},a_{22},\ldots,a_{2p_2} :
  \cdots :s_k;a_{k1},a_{k2},\ldots,a_{kp_k}|\nu \rangle$. First we use (i) to modify each of the $a_{ij}$ that lies outside the range $[0,q)$. To do so we first use the division algorithm to write $a_{ij} = u_{ij}q+a'_{ij}$ with $0\leq a'_{ij} <q$. For an arbitrary fixed $i,j$, we have:
$\langle \cdots : s_i; \ldots, a_{ij}, \ldots : \cdots \rangle =  \langle \cdots : s_i; \ldots, u_{ij}q+a'_{ij}, \ldots : \cdots \rangle \sim \langle \cdots : s_i-u_{ij}r; \ldots, u_{ij}q+a'_{ij}-u_{ij}q, \ldots : \cdots \rangle = \langle \cdots : s_i-u_{ij}r; \ldots, a'_{ij}, \ldots : \cdots \rangle$.

We repeat this process for all $i,j$. Relabelling $s_i-r\sum_{j=1}^{p_i}u_{i,j} = s^*_i$, we have $\langle \cdots : s^*_i; \ldots, a'_{ij}, \ldots : \cdots \rangle$ with all the $a'_{ij}$ in the range $[0,q)$.


In an analogous fashion we use (ii) to ensure that all the $s^*_i$ lie in the range $[0,q]$ by moving any excess onto $\nu$. By the division algorithm, for each $i$, $s^*_{i} = l_i q + s'_i$ with $0\leq s'_i<q$. Applying (ii) we obtain:
$\langle \cdots : s^*_i; \ldots :  \cdots | \nu \rangle =  \langle \cdots : l_iq+s'_i; \ldots :  \cdots | \nu \rangle \sim \langle \cdots : l_iq+s'_i-l_iq; \ldots :  \cdots | \nu-l_ir \rangle =  \langle \cdots : s'_i; \ldots :  \cdots | \nu-l_ir \rangle$.

We repeat this for all $i$. Then we relabel $\nu-r\sum_{i=1}^kl_i = \nu'$. This gives $\langle \cdots : s'_i; \ldots, a'_{ij}, \ldots:  \cdots | \nu' \rangle$, with $0\leq s'_i < q, 0 \leq a'_{ij} < q$, as desired.
\end{proof}

It follows from the above result that we need only consider finitely many combinations of values for $s_i$ and $a_{ij}$. Furthermore, for any fixed choice of $s_i$ and $a_{ij}$, there is only one possible value of $\nu$ that could allow formal satisfaction to occur. This value is defined by $h(1, \langle s_1;a_{11},a_{12},\ldots,a_{1p_1} : s_2;a_{21},a_{22},\ldots,a_{2p_2} :
  \cdots :s_k;a_{k1},a_{k2},\ldots,a_{kp_k} |\nu \rangle=0$. More explicitly, we can write \\$\nu=h(1, \langle s_1;a_{11},a_{12},\ldots,a_{1p_1} : s_2;a_{21},a_{22},\ldots,a_{2p_2} :
  \cdots :s_k;a_{k1},a_{k2},\ldots,a_{kp_k}|0 \rangle)$.

Together,  Theorems \ref{thm:nIsBounded},\ref{thm:invarianceOfH},
and \ref{thm:boundedParameters}, along with the above commentary on $\nu$, show that
given the parameters $k$ and $p_i$, $1 \le i \le k$, we need to check only finitely many parameter vectors
$\langle s_1;a_{11},a_{12},\ldots,a_{1p_1} :
  \cdots :s_k;a_{k1},a_{k2},\ldots,a_{kp_k}| \nu \rangle$, each for finitely many $n$,
to determine whether any ceiling function formally satisfies a recursion of the form (\ref{eq:generalRecurrence}). This still leaves a very large parameter space to investigate, which we accomplish using a computer.

Our program only checks for formal satisfaction. Thus, as we explained earlier, we need to demonstrate the conditions under which formal satisfaction implies generation as a unique solution sequence.

\begin{theorem}\label{thm:FormalGeneration}
Suppose that $\cln{rn}{q}$ formally satisfies a recursion $R(n)$ of the form ($\ref{eq:generalRecurrence})$ with all $a_{ij} > 0$. Further, assume that $rp_i/q \le 1$ for all $i$. Then there exists some recursion $R'(n)$ in the same equivalence class as $R(n)$ and some integer $m$ such that $\cln{rn}{q}$ is the unique solution generated by the recursion $R'(n)$ with the initial conditions $R'(1) = \cln{r}{q}, R'(2) = \cln{2r}{q}, \ldots, R'(m) = \cln{mr}{q}$. If we have $rp_i/q < 1$ for all $i$, then the above statement holds for $R(n)$ itself (with no need for $R'(n)$).
\end{theorem}
\begin{proof}
We begin by assuming that $rp_i/q < 1$ for all $i$. Consider the arguments $R(n-a_{ij})$. For $n$ large enough, $0< n-a_{ij} < n$ for all $i,j$. Thus, the desired result holds so long as for all $n$ sufficiently large, the arguments $n-s_i-\sum_{j=1}^{p_i}R(n-a_{ij})$ are all positive and less than $n$. This is equivalent to showing that for all $i$ and for large enough $n$, $\frac{n-s_i-\sum_{j=1}^{p_i}\cln{(n-a_{ij})r}{q}}{n}$ lies strictly between $0$ and $1$. Distributing the preceding fraction, we get $1-s_i/n-\sum_{j=1}^{p_i}\frac{\cln{(n-a_{ij})r}{q}}{n}$. But $s_i/n \rightarrow 0$, and $\frac{\cln{(n-a_{ij})r}{q}}{n} \rightarrow r/q$, so $\frac{n-s_i-\sum_{j=1}^{p_i}\cln{(n-a_{ij})r}{q}}{n} \rightarrow 1-p_ir/q$, which fulfills the requirement.

Now we deal with the case where for some $i$, $p_ir/q = 1$. In this situation we have only that $\frac{n-s_i-\sum_{j=1}^{p_i}\cln{(n-a_{ij})r}{q}}{n} \rightarrow 0$, so we still need to prove that the argument of the $i^{th}$ term is always eventually positive. To do so, we transform $R(n)$ with an equivalent recursion $R'(n)$ as follows. First, using relation (i), we replace all $a_{ij}$ with $a_{ij}'$ such that $0 < a_{ij}' \leq q$ and move any excess onto $s_i$, which we rename $s_i^*$. Next, we use relation (ii) to change $s_i^*$ into $s_i'$ such that $s_i' < -2p_i$, and move the necessary amount onto $\nu$, renaming the new constant term $\nu'$. Now, the $i^{th}$ recursion summand is $R'(n-s_i'-\sum_{j=1}^{p_i}R'(n-a_{ij}'))$. Because $0< a_{ij}' \leq q$, we know that $|R'(n-a_{ij}')-rn/q| < 2$, so $n-s_i'-\sum_{j=1}^{p_i}R'(n-a_{ij}')) > n - s_i' - n - 2p_i$. But since $s_i' < -2p_i$, we have $n-s_i'-\sum_{j=1}^{p_i}R'(n-a_{ij}')) > 0$ as required. In this way, we have prevented the $i^{th}$ summand of the recursion from having a negative argument. We can repeat as needed.
\end{proof}

Recall that if $\cln{rn}{q}$ formally satisfies a recursion of the form ($\ref{eq:generalRecurrence})$ then $r/q = \frac{k-1}{\sum_{i=1}^{k}{p_i}}$. Once $k$ and $p_i$ are fixed we showed in Theorems \ref{thm:boundedParameters} and \ref{thm:nIsBounded} above that we only need to check finitely many equalities to determine if $\cln{rn}{q}$ satisfies any recursions with the given $k$ and $p_i$. With Theorem \ref{thm:FormalGeneration}, we know that the only recursions of interest have $p_i \le q/r$ for all $i$. Thus, there are only finitely many possible combinations of $k$ and $p_i$ that can yield a given ceiling function $\cln{rn}{q}$ with $r>1$ as a unique solution.

For example, if $\cln{2n}{5}$ solves some recursion $R(n)$ of the form ($\ref{eq:generalRecurrence})$, then it must be the case that all $p_i \leq 5/2$. This leaves only $k=3, p_1=1, p_2=p_3 = 2$ and $k=5, p_i=2$. By using the approach described earlier in this chapter, for each such combination of $k$ and $p_i$, we can determine conclusively whether there are any choices of $s_i, a_{ij},$ and $\nu$ that do in fact generate the desired ceiling function.

Before we turn to our empirical findings in the next section, we note that there are two additional considerations that allow us to further reduce the parameter space that we must investigate. First, if $\cln{n}{q}$ satisfies some recursion for a given $k_a$ and set of $p_i$ ($1 \leq i \leq k_a$), then $\cln{n}{q}$ also satisfies a recursion with $k'=k_a+1$, the same $p_i$ for $1 \leq i \leq k_a$, and $p_{k'} = q$. For example, since we know (see \cite{BLT}) that $\cln{n}{2}$ solves a recursion with $k=2$ and $p_1=p_2 = 1$, we also know that $\cln{n}{2}$ satisfies a recursion with $k=3$ and $p_1=p_2 = 1, p_3 = 2$, with $k=4$ and $p_1=p_2=1, p_3=p_4=2$, and so on. More precisely:

\begin{theorem}
\label{thm:extraq}
Suppose $\cln{n}{q}$ formally satisfies the recursion $R(n) = \sum_{i=1}^kR(n-s_i-\sum_{j=1}^{p_i}R(n-a_{ij})) + \nu$. Then $\cln{n}{q}$ also formally satisfies the recursion $R'(n) = \sum_{i=1}^kR'(n-s_i-\sum_{j=1}^{p_i}R'(n-a_{ij})) + R(n-1-R(n-1)-R(n-2)-\ldots-R(n-q)) + \nu$.
\end{theorem}
\begin{proof}
The result follows by substituting $\cln{n}{q}$ into the argument of the $(k+1)^{st}$ term in the recursion. Define $f(n):=n-1-\cln{n-1}{q}-\cln{n-2}{q}-\ldots-\cln{n-q}{q}$ for all $n$. Then $f(0) = 0-1-0-0-\ldots-0+1=0$. Next, note that $f(n+1)=f(n)$ since exactly one of the terms $\cln{n-i}{q}$ will increase by 1 when $n$ increases by $1$. Thus, $f(n)=0$ for all $n$, so the $(k+1)^{st}$ summand of $R'(n)$ is always 0.
\end{proof}

Note that generation as a solution sequence for some recursion equivalent to $R'$ follows from Theorem \ref{thm:FormalGeneration}.

Second, if $\cln{n}{q}$ satisfies a recursion with a given $k$ and set of $p_i$, $1 \leq i \leq k$, then $\cln{n}{mq}$ satisfies a recursion with the same $k$ and $p_i'=mp_i$ for $1 \leq i \leq k$. This result is an easy extension of the Order Multiplying Interleaving Theorem proved in \cite{ConollyLike} (Theorem 4.2). We omit the details. For example, since we know (see \cite{BLT}) that $\cln{n}{2}$ solves a recursion with $k=2$ and $p_1=p_2 = 1$, we also know that there are recursions satisfied by $\cln{n}{4}$ with $k=2$ and $p_1=p_2 = 2$, by $\cln{n}{6}$ with $k=2$ and $p_1=p_2=3$, and so on. We omit these entries from the empirical results we present in the following section.

\section{Which Ceiling Functions Occur as Solutions?} \label{sec:Empirical}
Using the computer verification process described above, we are able to identify certain ceiling functions that appear as solutions to recursions and rule out others that cannot occur.
These results are summarized in Table \ref{tab:ClnFunSolutions}.

\begin{table}
\caption{Which ceiling functions occur as solutions?}
\begin{center}
{\begin{tabular}{| c | ccccccc | ccccccc |}
\hline
   & \multicolumn{6}{c}{Recursions exist} &
   & \multicolumn{6}{c}{Recursions do not exist} & \\\hline
$r/q$                     &
$p_1$ & $p_2$ & $p_3$ & $p_4$ & $p_5$ & $p_6$ & $p_7$ &
$p_1$ & $p_2$ & $p_3$ & $p_4$ & $p_5$ & $p_6$ & $p_7$ \\\hline

$1/2$ &
  1   &   1   &       &       &       &       &       &
      &       &       &       &       &       &       \\\hline

$1/3$ &
  2   &   2   &   2   &       &       &       &       &
  1   &   2   &       &       &       &       &       \\\hline

$1/4$ &
  2   &   3   &   3   &       &       &       &       &
  1   &   3   &       &       &       &       &       \\\cline{2-15}

      &
  3   &   3   &   3   &   3   &       &       &       &
      &       &       &       &       &       &       \\\hline

$1/5$ &
  4   &   4   &   4   &   4   &   4   &       &       &
  1   &   4   &       &       &       &       &       \\\cline{2-15}

      &
      &       &       &       &       &       &       &
  2   &   3   &       &       &       &       &       \\\cline{2-15}

      &
      &       &       &       &       &       &       &
  3   &   3   &   4   &       &       &       &       \\\cline{2-15}

      &
      &       &       &       &       &       &       &
  2   &   4   &   4   &       &       &       &       \\\hline

$1/6$ &
  5   &   5   &   5   &   5   &   5   &   5   &       &
  1   &   5   &       &       &       &       &       \\\cline{2-15}

      &
      &       &       &       &       &       &       &
  2   &   4   &       &       &       &       &       \\\hline

$1/7$ &
  6   &   6   &   6   &   6   &   6   &   6   &   6   &
  1   &   6   &       &       &       &       &       \\\cline{2-15}

      &
      &       &       &       &       &       &       &
  2   &   5   &       &       &       &       &       \\\cline{2-15}

      &
      &       &       &       &       &       &       &
  3   &   4   &       &       &       &       &       \\\hline

$2/3$ &
      &       &       &       &       &       &       &
  1   &   1   &   1   &       &       &       &       \\\hline

$2/5$ &
      &       &       &       &       &       &       &
  1   &   2   &   2   &       &       &       &       \\\hline

$3/4$ &
      &       &       &       &       &       &       &
  1   &   1   &   1   &   1   &       &       &       \\\hline

$4/5$ &
      &       &       &       &       &       &       &
  1   &   1   &   1   &   1   &   1   &       &       \\\hline
\end{tabular}}
\end{center}
\label{tab:ClnFunSolutions}
\end{table}

Observe first that ceiling functions of the form $\cln{n}{q}$ always solve at least one recursion for all the cases that we checked. In particular, such functions solve recursions with $k=q$ and all $p_i=q-1$. This empirical observation led us to the following result:

\begin{theorem}
For all $q>1$, the ceiling function $\cln{n}{q}$ formally satisfies $\langle 0;1,2,\ldots,q-1:1;1,2,\ldots,q-1:\ldots q-1;1,2,\ldots,q-1|0 \rangle$.
\end{theorem}
\begin{proof}
By Theorem $\ref{thm:nIsBounded}$, we need only check that $\cln{n}{q} = \sum_{i=0}^{q-1}\cln{n-i-\sum_{j=1}^{q-1}\cln{n-j}{q}}{q}$ for $1 \leq n \leq q^2$. By the division algorithm, write $n = mq+d$ with $1 \leq d \leq q$ (observe that we exclude $d=0$ and instead allow $d=q$); by our assumption on $n$, we have $0 \leq m \leq q-1$. Now, $\cln{n}{q} = m+1$, while $\cln{n-j}{q}$ is equal to $m$ if $j \geq d$ and $m+1$ if $j<d$.

Thus, we have $\sum_{i=0}^{q-1}\cln{n-i-\sum_{j=1}^{q-1}\cln{n-j}{q}}{q} = \sum_{i=0}^{q-1}\cln{mq+d-i-m(q-1)-\sum_{j=1}^{q-1}[[j<d]]}{q}$, where we use the usual Iversonian notation, $[[P]] = 1$ if $P$ is true and $0$ if $P$ is false. Simplifying the latter sum, we get $\sum_{i=0}^{q-1}\cln{m+d-i-(d-1)}{q} = \sum_{i=0}^{q-1}\cln{m+1-i}{q} = \sum_{i=0}^{q-1}[[i<m+1]] = m+1$. This concludes the proof.
\end{proof}

By Theorem \ref{thm:FormalGeneration} we conclude that the recursion $\langle 0;1,2,\ldots,q-1:1;1,2,\ldots,q-1:\ldots q-1;1,2,\ldots,q-1| 0 \rangle$ with sufficiently many initial conditions generates $\cln{n}{q}$ as its unique solution.

In fact, we have considerable empirical evidence to suggest that $\cln{n}{q}$ solves many other such recursions. From this evidence we conjecture the following:

\begin{conjecture}
If $\cln{n}{q}$ satisfies the recursion $\langle s_1;a_{11},a_{12},\ldots,a_{1p_1} : s_2;a_{21},a_{22},\ldots,a_{2p_2} :
  \cdots :s_k;a_{k1},a_{k2},\ldots,a_{kp_k}| 0 \rangle$ with $k=q$ and $p_i=q-1$, then $q$ does not divide any $a_{ij}$ and $q\sum_{i=1}^k s_i =\sum_{i=1}^k \sum_{j=1}^{p_i} a_{ij}$.
\end{conjecture}

Note that we required $\nu=0$ in the above conjecture. However, by the second equivalence class relation we can always find an equivalence class representative with $\nu=0$ since $r=1$.

Our computer verification confirms that none of the ceiling functions $\cln{2n}{3}, \cln{3n}{4}, \cln{3n}{5}$ or $\cln{4n}{5}$ can occur as the solution to any recursion of the form (\ref{eq:generalRecurrence})\footnote{Note that $\cln{3n}{5}$ has no possible choice of $k$ and $p_i$, so does not appear on the table. The computer verification for $\cln{2n}{5}$ with $k=5$ and all the $p_i=2$ takes too long to be practical. A partial search suggests that no recursions solved by $\cln{2n}{5}$ exist.}. We have not checked other ceiling functions with larger denominators because of the amount of computing power required.\footnote{For fixed values of $k$ and $p_i$, our algorithm must consider $\flr{q^2}{r} r^k q^{k-2 + \sum_{i=1}^k p_i}$ vectors $  y \in \mathbb{Z}^{k+\sum_{i=1}^{k} p_{i}}$, and up to $q^2$ values must be computed for each of these.} Based on this preliminary finding we conjecture the following:

\begin{conjecture}
\label{conj:roverq}
If $\cln{rn}{q}$ is the solution generated by a recursion of the form (\ref{eq:generalRecurrence}), then $r=1$.
\end{conjecture}

Note that if the above conjecture is true, then the variable $\nu$ is superfluous, since we could always find a equivalence class representative with $\nu=0$.

\section{Conway Connections} \label{sec:Conway}
Ceiling functions provide an unexpected avenue for identifying a novel connection between generalized Conolly recursions of the form (\ref{eq:generalRecurrence}) and a natural generalization of the famous Conway-Hofstadter sequence.  Recall that the Conway sequence $C(n)$ \cite{Conway, KuboVakil, Pelesko}, is defined by the nested recursion $C(n) = C(C(n-1))+C(n-C(n-1))$, with $C(1)=C(2)=1$. Analogous to (\ref{eq:generalRecurrence}), which generalizes the Conolly recursion \cite{ConVa}, the most general form of the Conway recursion is

\begin{align}
R(n) = \sum_{i=1}^{k_1}R(n-s_i-\sum_{j=1}^{p_i}R(n-a_{ij}))+\sum_{i=1}^{k_2}R(-t_i+\sum_{j=1}^{q_i}R(n-b_{ij}))+\nu,
\label{eq:ConwayTerms}
\end{align}
where if $k_1, k_2, q_i,$, or $p_i = 0$, then we take the corresponding sum to be the empty sum (ie, 0). 

A priori there is no reason to believe that there ought to be any connection between these two families of recursions. However, through experimental evidence, we observed that in many cases recursions of the form (\ref{eq:ConwayTerms}) have ceiling function solutions just like those of the form (\ref{eq:generalRecurrence}). In what follows, we provide the first known results linking the solutions of these two families of recursions by establishing an equivalence between them in certain circumstances.

To begin, we introduce some terminology and notation.
We say a \textit{Conway term} of a recursion is one of the form $R(-a+R(n-b))$, and a \textit{Conolly term} is one of the form $R(n-a-R(n-b))$. As it turns out, for the purpose of ceiling function solutions of the form $\cln{n}{q}$, certain Conway terms and Conolly terms have a reciprocal relationship wherein one can be replaced by the other for the purposes of a given ceiling function solution. This would not work for ceiling function solutions of the form $\cln{rn}{q}$ with $r>1$ (although we don't believe that any such solutions exist; see Conjecture \ref{conj:roverq}).

\begin{theorem}
Let $I$ be a subset of $\{0,1,...,q-1\}$. The ceiling function $\cln{n}{q}$ formally satisfies the recursion $R(n) = \sum_{i=1}^{k_1}R(n-s_i-\sum_{j=1}^{p_i}R(n-a_{ij}))+\sum_{i=1}^{k_2}R(-t_i+\sum_{j=1}^{q_i}R(n-b_{ij}))+R(n-a-\sum_{j \in I}R(n-j))+\nu$ if and only if it formally satisfies the reciprocal recursion $R(n) = \sum_{i=1}^{k_1}R(n-s_i-\sum_{j=1}^{p_i}R(n-a_{ij}))+\sum_{i=1}^{k_2}R(-t_i+\sum_{j=1}^{q_i}R(n-b_{ij}))+R(-a+\sum_{j \notin I}R(n-j))+\nu$.
\end{theorem}
\begin{proof}
Observe that as in the proof of Theorem \ref{thm:extraq}, $n = \sum_{j=0}^{q-1}\cln{n-j}{q}$. This shows that the argument of the $(k_1+k_2+1)^{st}$ term is the same in both recursions.
\end{proof}

By using the equivalence relations (i) and (ii) introduced in Section \ref{sec:Alg}, we can always transform a Conolly term so that all the $a_{ij}$ are in the range $0 \le a_{ij} <q$. In order to apply the above theorem, it is necessary that for the index $i$ corresponding to the summand that we wish to transform, we have all of the $a_{ij}$ distinct modulo $q$. Likewise, we could easily derive equivalence relations for the Conway terms that allow us to do the same thing (they would be similar to the equivalence relations in Section \ref{sec:Alg} for the Conolly terms, differing only by a sign change in relation (i)). These would allow us to transform a Conway term into the corresponding Conolly term, provided that for the chosen index $i$, the $b_{ij}$ are all distinct modulo $q$.

For example, this means that the recursion $R(n) = R(n-R(n-1))+R(n-1-R(n-1))$, which is solved by $\cln{n}{2}$, can have either or both Conolly terms replaced by a Conway term. Thus, $C(n) = C(n-C(n-1)) + C(-1+C(n))$ would also be solved (formally) by $\cln{n}{2}$. Applying an equivalence relation to the second term, we get that $\cln{n}{2}$ solves the recursion $C(n)=C(n-C(n-1)) + C(C(n-2))$.

\section{Conclusions and Future Directions} \label{sec:Conc}

Several directions exist for future research on this topic. The first would be to follow up on the two conjectures in Section \ref{sec:Empirical}, as well as to further explore patterns among the parameter sets for the solutions that we obtained; gathering additional empirical evidence would be a useful initial step.

Second, the link between the generalized Conolly and Conway recursions in the case of ceiling function solutions suggests that there might be other connections between these two families of recursions. For example, do the tree interpretations for solutions to the Conolly-type recursions found in \cite{JacksonRuskey, Rpaper, ConollyLike} have any analogue for Conway-type recursions?

Third, the Conway connection we found handles all ceiling function solutions to the Conway generalization provided that the $b_{ij}$ in each Conway term are distinct mod q. When this is not the case, an approach similar to the one we used in Section \ref{sec:Empirical} for the generalized Conolly recursion (\ref{eq:generalRecurrence}) might provide interesting results.

Finally, ceiling and floor functions appear frequently in the solutions of a wide variety of nested recursions, and not always in the form $\cln{n}{q}$ that we found above. For example, the Golomb recursion \cite{Golomb1990} $G(n)=G(n-G(n-1))+1$, with $G(1)=1$, has the closed form solution $\lfloor\frac{\lfloor\sqrt{8n}\rfloor+1}{2}\rfloor$. The function $\lfloor n\alpha^{(k)}\rfloor$, where $\alpha^{(k)}=\frac{2-k+\sqrt{k^2+4}}{2}$, formally satisfies another recursion defined by Golomb \cite{Golomb1990}, namely,  $2b(n)+kn=b(b(n)+kn)$; see also \cite{BCT}. Does this suggest that ceiling and floor functions may be a unifying theme in the study of nested recursions? Can any of the techniques in this paper be generalized to a broader class of ceiling functions?

\end{document}